\documentclass[11pt, reqno]{amsart}
\usepackage{amsmath}
\title[upper graph box dimension of bounded subsets]{On a conjecture regarding the upper graph box dimension of bounded subsets of the real line}
\newtheorem{theorem}{Theorem}
\newtheorem{corol}{Corollary}

\newtheorem{lemma}{Lemma}
\newtheorem{remark}[theorem]{Remark}

\newcommand{\IN}{\mathbb{N}}
\newcommand{\IR}{\mathbb{R}}

\newcommand{\gdims}[1]{\overline{dim}_{gr,B}(#1)}
\newcommand{\dims}[1]{\overline{dim}_{B}(#1)}
\newcommand{\atopo}{this is impossible}

\begin{document}

\begin{abstract}
Let $X\subset \IR $ be a bounded set; we introduce a formula that calculates the upper graph box dimension of $X$ (i.e. the supremum of the upper box dimension of the graph over all uniformly continuous functions defined on $X$). We demonstrate the strength of the formula by calculating the upper graph box dimension for some sets and by giving an "one line" proof, alternative to the one given in \cite{initial}, of the fact that if $X$ has finitely many isolated points then its upper graph box dimension is equal to the upper box dimension plus one. Furthermore we construct a collection of sets $X$ with infinitely many isolated points, having upper box dimension $a$ taking values from zero to one while their graph box dimension takes any value in $[\max\{2a,1\},a+1],$ answering this way, negatively to a conjecture posed in \cite{initial}.
\end{abstract}
\maketitle
\section{Introduction}

Let $X$ be a set, let also $C_{u}(X)$ be the set of all uniformly continuous functions on $X$ equipped with the uniform norm $\|\cdot\|_{\infty}$ . In \cite{initial}, the concept of the upper graph box dimension was introduced, i.e.
$$\gdims{X}=\sup_{f\in C_{u}(X)}\dims{graph(f)} $$
and it was proved, that a \textbf{typical} element (in the sense of Baire) in the set $C_{u}(X)$, has a graph with upper box dimension equal to the upper graph box dimension of the set. To put it in another way, it was proved that a typical element in $C_{u}(X)$ has a graph with upper box dimension as high as allowed by the set. The proof given, made no use of any properties of the set $X$. It was in the lines of "if the set $X$ can accommodate a function having graph with upper box dimension bigger or equal than $b$, then a typical function will do as well". In \cite{initial}, it was proved that 
\begin{equation}\label{1} 1\leq\gdims{X}\leq\dims{X}+1 \end{equation} and for the case where $X$ has finitely many isolated points, it was proven that $$\gdims{X}=\dims{X}+1, $$ while the general case remained open. It was conjectured that the upper graph  box dimension of $X$ is either equal to the upper box dimension of $X$ plus one or just one.

In this paper, we introduce a formula that calculates the upper graph box dimension of a set $X$. By using this formula we refine inequality \ref{1}, we give a straightforward alternative proof of the fact, that if $X$ has finitely many isolated points then its upper graph box dimension is equal to the upper box dimension plus one and even more we use the formula to calculate the upper graph box dimension for a collection of natural sets. We conclude by constructing a collection of sets having all possible values allowed by the refined inequality, and disproving this way the conjecture in \cite{initial}.
\begin{remark}
In \cite{Williams} and \cite{Petruska}, it was respectively proved that when $X=[0,1],$ a typical function in $C_{u}(X)$ has a graph with Hausdorff dimension equal to 1 and packing dimension equal to 2.
Although we are not aware of any extensions of these results in general sets $X$, we strongly believe that for a typical  element in $C_{u}(X)$ we always have $\dim_{H}(graph(f))=\dim_{H}(X)$ and $\dim_{p}(graph(f))=\dim_{p}(X)+1$, and therefore a concept like the upper graph dimension is useful only for the box dimension.
\end{remark}

For simplicity we will assume that $X\subset[0,1]$. We start by recalling the definition of the upper box dimension of subsets of $\IR^{d}$. For $\delta>0$, let \begin{equation}\mathcal{Q}_{\delta}^{d}=\left\{\prod_{i=1}^{d}[n_{i}\delta,(n_{i}+1)\delta]\Bigg|n_{1},...,n_{d}\in Z\right\} \end{equation}
denote the standard $\delta-$grid in $\IR^{d}$, and for a subset $X$
of $\IR^{d}$ we write
\begin{equation}
N_{\delta}(X)=\Bigg|\left\{ Q\in\mathcal{Q}_{\delta}^{d} \big| Q\cap X\neq 0\right \}\Bigg|
\end{equation}
for the number of cubes in $\mathcal{Q}_{\delta}^{d} $ that intersects $X$. The upper box dimension of $X$ is now defined by
\begin{equation}
\dims{X}=\limsup_{\delta\rightarrow 0}\frac{\log{N_{\delta}(X)}}{-\log\delta}.
\end{equation}
The reader is referred to Falconer's \cite[p. 42]{Falconer} for a thorough  discussion on the properties of the box dimension. One property that we are going to use here, regards the alternative type of boxes that can be used in the definition. More specifically, we will be working with $ \delta $-meshes of \textbf{disjoint} cubes of the form $ [m_1\delta,(m_1+1)\delta) \times [m_2\delta,(m_2+1)\delta) $. 

Also note that (\cite{Falconer}) it is enough to consider limits as $ \delta $ tends to 0 through any decreasing sequence $ \delta_k $ - as long as $ \delta_{k+1} \ge c\delta_k $. Taking $ \delta_k = \frac{1}{k} $ we can work with limits of $ \frac{\log N_{\frac{1}{k}}(F)}{\log k} $ as $ k \in \IN $ tends to infinity.

For $f\in C_{u}(X)$, we will write graph$(f)$ to denote the graph of $f$, ie.
$$\text{graph}(f)=\{(x,f(x))|x\in X\}.$$
With a slight abuse of notation, we are going to write $N_{\delta}(f)$ instead of $N_{\delta}(\text{graph}(f))$.

With $P(X)$ we are going to define all the polygonic functions restricted in $X$.
 
Finally we define the sequence
$$g_{m}(X)=\sum_{k=1}^m \min\{m,\#(X \cap [\frac{k-1}{m},\frac{k}{m}])\}.$$

\section{equivalence of definitions and applications of the formula}
In the first part, we are going to prove that we can use $g_{m}$ to calculate the upper graph box dimension of a set. More specifically we have 

\begin{theorem}
\label{theorem001}
\[
\gdims{X}=\limsup_{m\to \infty}\frac{\log(g_{m})}{\log(m)}.
\]
\end{theorem} 
 Afterwards we are giving an alternative proof to the fact, that if $X$ has finitely many isolated points, then $$\gdims{X}=\dims{X}+1 .$$ 

We conclude the section, by providing some natural examples of sets where $g_{m}$ can be used to calculate their upper graph box dimension.

\begin{lemma}
\label{lemma002}
Let  $f \in P(X),  \delta_{0} > 0, p > 0$ and  $\epsilon > 0.$ It exists $ \delta_{1} : \delta_{0} > \delta_{1} > 0 $ and $ g \in P(X)$ with $g > 0 $, such that $ ||g||_\infty < p $ and
\[
 \frac{\log N_{\delta_{1}}(g+f)}{-\log \delta_{1}} \geq \limsup_{m \to \infty} {\frac{\log g_m}{\log m}}  - \epsilon
\] 
\end{lemma}
\begin{proof}
  If $ a = \displaystyle \limsup_{m \to \infty} \frac {\log g_m}{\log m} $ then there exists a subsequence $ g_{m_k} $ with $m_k \to \infty $ such that $ \displaystyle a = \lim_{m_k \to \infty} \frac {\log g_{m_k}}{\log m_k} $. That means that $ \forall \epsilon >0$ there exists a $ k_0\in N $ such that
  
  \begin{equation}
    \label{lema002sxesi1}
    \forall m_k > m_{k_0} \Rightarrow | a - \frac {\log g_{m_k}}{\log m_k} | < \frac{\epsilon}{2}.
  \end{equation}
  
  Let now arbitrary $\delta > 0 $ of the form $\frac{1}{m}$. We will construct a function $ g_{\delta,p} \in P(X) $ as follows. 
For every interval $ I_k = [\frac{(k-1)}{n},\frac{k}{m}] $ we select $ n_{k} = min \{mp, \# X \cap I_k) \} $ elements of $ I_k \cap X $, we will call $ b^{k}_i , i\in\{1,...,n_{k}\}.$  We do this $ \forall k : I_k \cap X \neq \emptyset $ and we end up with a finite subset of $ X $. For every $ b^{k}_i $ we define a point $ (b^{k}_i,g_1(b^{k}_i)) $ in a way that, no two points occupy the same box and $g_1(b^{k}_i) < p $ for all $k\in \{1,...,m\},i\in\{1,...,n_{k}\}.$ If we consider $ g_1 $ to be the polygonal line joining all $ (b^{k}_i,g_1(b^{k}_i))$ then  $ g_{\delta,p} $ is the restriction of $ g_1 $ in $ X $. From its construction $ || g_{\delta,p} || < p $ and $ N_\delta(g_{\delta,p}) \geq \sum_{k=1}^{m} \min\{pm,\#X \cap [\frac{(k-1)}{m},\frac{k}{m}]\} .$
  
   Since $f$ is polygonic it satisfies the assumptions of lemma \ref{polyf} for some constant $c,$ and therefore from lemmas \ref{boxes_of_f+g} and \ref{polyf} we have that $ N_\delta(f+g_{\delta,p}) \geq \frac{N_\delta(g_{\delta,p})}{2c}.$
 
We have:
   \begin{align*}
    N_{\delta}(f+g_{\delta,p}) \geq \frac{N_{\delta}(g_{\delta,p})}{2c} &\geq c' \sum_{k=1}^{m} \min\{pm,\#(X \cap [(k-1)/m,k/m])\}  \\
   & \geq c'' \sum_{k=1}^{m} \min\{m,\#(X \cap [(k-1)/m,k/m])\}.
  \end{align*}
  
Where $ c'' $ depends on $ p .$
Using the above we have
  \begin{align*}
    \frac{\log \left(N_{\delta}(f+g_{\delta,p})\right)}{\log \delta} & \geq 
    \frac {\log \left(c'' \sum_{k=1}^{m} \min\{m,\#X \cap [(k-1)/m,k/m]\} \right)}{\log \delta}  \\
    &=\frac{\log c''}{\log m} + \frac{\log\left(\sum_{k=1}^{m} \min\{m,\#X \cap [(k-1)/m,k/m]\}\right)}{\log m}.
  \end{align*}

  We can select a sufficiently large $ m_0 $, such that for $ m > m_0$ we get $\frac{\log c''}{log m} > - \epsilon / 2 $.
  Now if we select $ \delta_{1}=\frac{1}{m_{1}} $ with $ m_{1} \in \{ m_k \}_{k\in N}$ satisfying $$ m_{1} > \max \{m_{k_0}, m_0 ,\frac{1}{\delta_{0}}\},$$ we will have that:
  
  \begin{align*}
 \frac{\log N_{\delta_{1}}(g_{\delta_{1},p}+f)}{-\log \delta_{1}} & = \frac{\log N_{\delta_{1}}(g_{\delta_{1},p}+f)}{\log m_1} \\ & 
\geq\frac{\sum_{k=1}^{m_1} \min \{ m_1,\# X \cap [(k-1)/m_1,k/m_1] \} }{\log m_{1}} - \frac{\epsilon}{2} \\ &\stackrel{(\ref{lema002sxesi1})}{\geq}
    \limsup_{m \to \infty} \frac {\log g_m}{\log m} - \frac{\epsilon}{2} - \frac{\epsilon}{2} = \limsup_{m \to \infty} \frac {\log g_m}{\log m} - \epsilon .
  \end{align*}

\end{proof}

\begin{proof}[Proof of Theorem 1.]

  First we will show that
  \[
    \gdims{X} \geq \limsup_{m \to \infty} \frac{\log g_m}{\log m}
  \]
  For simplicity let $ a = \limsup_{m \to \infty} \frac{\log g_m}{\log m} $. Let $ f_1(x) = \frac{1}{4}x, F_{1}=f_{1}, \delta_{1}=\frac{1}{2}$ and. Let also assume that for $i\in\{1,2...,n-1\},$ we have chosen $f_{i},$ $F_{i},$ and $\delta_{i}$ to satisfy 
  \begin{enumerate}
\item $F_{i}=\sum_{j=1}^{i}f_{j} ,$
\item $\frac{\log N_{\delta_{i}}(F_{i})}{-\log{\delta_{i}}}\geq a-\frac{1}{i},$
\item $\delta_{i}<\max\{\frac{1}{i},\delta_{i-1}\},$
\item $ || f_i ||_{\infty} \leq  min \{ \frac{\delta_1}{2^i},\frac{\delta_2}{2^{i-1}}, \dots, \frac{\delta_i}{2}, \frac{1}{2^i} \}.$
\end{enumerate}

 By Lemma \ref{lemma002} for $g=F_{n-1},\delta=\delta_{n-1}, 
p = min \{ \frac{\delta_1}{2^{n-1}},\frac{\delta_2}{2^{n-2}}, \dots, \frac{\delta_{n-1}}{2}, \frac{1}{2^{n}} \}$ and $\epsilon=\frac{1}{n},$ 
we find $ f_n \in P(X), \delta_{n}>0, $ with $ || f_n ||_{\infty} \leq p$ and  $ 
\delta_{n} < min \{ \frac{1}{n},\delta_{n-1} \} $ such that $ \frac{\log 
N_{\delta_{n}}(F_n)}{-\log \delta_{n}} \geq a - \frac{1}{n}. $

Since $|| f_n ||_{\infty} \leq \frac{1}{2^{n}},$ we have that $ F_i $ converges uniformly to some $ F \in C_u(X). $   
 
Also for $n\in N,$ since $ ||f_{i}||<\frac{\delta_{n}}{2^{i-n+1}} ,\forall i>n$ we have $ \sum_{i=n+1}^{\infty}||f_{i}||<\delta_{n} $ and therefore by Lemma \ref{lemma001} $$\frac{\log N_{\delta_n}(F)}{-\log \delta_n} \geq \frac{\log \frac{1}{2} N_{\delta_n}(F_n)}{-\log \delta_n} = \frac {\log 2}{\log \delta_n} + \frac {N_{\delta_n(F_n)}}{\log \delta_n} \geq a - \frac{1}{n} +\frac {\log 2}{\log \delta_n}. $$ So we will have that $\lim_{\delta_n \to 0} \frac{\log N_{\delta_n}(F)}{-\log \delta_n} \ge a .$ Since $ F \in C_u(X) $ that means that $ \gdims{X} \geq a. $
  
  Now we will show that $ \gdims{X} \leq \lim_{m \to \infty} \frac{\log g_m}{\log m} $.
  
  Again $ a = \limsup_{m \to \infty} \frac{\log g_m}{\log m} $. It is obvious that $ \forall m, \forall f ~ N_{\frac{1}{m}}(f) \leq g_m $. That means that for every $ f ~,~ \limsup_{m \to \infty} \frac{\log N_{\frac{1}{m}}(f)}{\log m} \leq a $. Since we can simply take limits for $ \delta_m = \frac{1}{m} $ we have 
  \[
    \gdims{X} = \sup_{f \in C_u(X)} \limsup_{\delta \to 0}  \frac{\log N_{\delta}(f)}{-\log \delta} = \sup_{f \in C_u(X)} \limsup_{m \to \infty} \frac{\log N_{\frac{1}{m}}(f)}{\log m} \leq a.
  \]
  
\end{proof}

\begin{corol}\label{corol1}
  Let $X$ be a subset of [0,1] with finitely many isolated points. Then
  \[
    \gdims{f} = \dims{X} + 1.
  \]
  
\end{corol}
\begin{proof}
 If a set has finitely many isolated points we may remove those without affecting the box dimensions of the set. So every point in $ X $ can be considered an accumulation point. So we will have $ \min\{m,\#(X \cap [\frac{k-1}{m},\frac{k}{m}])\} = m $ for at least half the boxes that intersect with $ X $ - the half are taken to account for edge behavior. That gives
  \begin{align*}
    \gdims{K} = \limsup_{m \rightarrow \infty} \frac{\log g_m}{\log m} = \limsup_{m \rightarrow \infty} \frac{\log \frac{1}{2} m N_{1/m}(K)}{\log m} = \\
    1 + \limsup_{m \rightarrow \infty} \frac{N_{1/m}(K)}{\log m} = 1 + \dims{K}.
  \end{align*} 

\end{proof}

\begin{corol}
\label{corol01}
Let $ A = \{a_n\}~,~a_n = \dfrac{1}{n^p}.$ We have
$\dims{A}=\frac{1}{p+1}$ and $\gdims{A}=\frac{2}{p+1}$
\begin{proof}

For $ f(x) = \displaystyle{\frac{1}{x^{p}}} ~,~ x > 0 $ we have
\[ 
  f'(x) = -p x^{{(-p - 1)}} < 0 \text{ and } f''(x) = {(p + 1)} p
  x^{{(-p - 2)}} > 0
\]
From the relationship $ f(k) - f(k+1) = (k - k - 1) \cdot f'(u) = -f'(u) $, for
$ u \in (k,k+1) $ and the 
fact that $ |f'(x)| = -f'(x) $ is a decreasing function we have that $ b_n = a_n
- a_{n+1} $ is a decreasing sequence.
That means that for the smallest $ n_0 $ such that $ f'(n_0) < 1/m $ we have
\begin{equation}
\label{n0rel}
  \forall n > n_0 \Rightarrow a_{n}-a_{n+1} < \frac{1}{m} \text{ and } \forall n
<
  n_0 \Rightarrow 
  a_{n}-a_{n+1} > \frac{1}{m}.
\end{equation} \\
This tells us that for $ n \leq  n_0 $ we cannot have two distinct $ a_n $ in
the same box (that would mean their 
distance is $ < \frac{1}{m} $). Likewise for $ n > n_0 $ we cannot have a box
with no element of $ \{a_n\} $ in it 
(that would mean we have a distance that is $ > \frac{1}{m} $). \\
So to ``count'' the number of boxes that have elements of $ \{a_n\} $ inside
them all we need to do is find $ n_0 $, find
which box $ a_{n_0} $ lies in and add $ n_0 - 1 $ to the number of that box.
That means we are counting all boxes that are 
closer to $ 0 $ than the box $ a_{n_0} $ is in (including that box) and we are
counting one box for every element of our 
sequence before $ n_0 $. \\[6pt]
This gives us

\begin{align*}
  f'(x) = \frac{1}{m} \\
  -p x^{{(-p - 1)}} = \frac{1}{m} \\
  x = \sqrt[p+1]{mp} \\
  n_0 = \left [ \sqrt[p+1]{mp}~ \right ]
\end{align*}

Since $  \frac{k-1}{m} < a_{n_0} \leq \frac{k}{m} \Rightarrow k-1 < ma_{n_0}
\leq k \Rightarrow a_{n_0} $
lies in box number $  [ma_{n_0}] $. \\
To calculate the dimension of set $ A $ we must now calculate 

\begin{equation}
  \lim_{m \to \infty} \frac{\log (n_0 + [ma_{n_0}] )}
  {\log m} = \lim_{m \to \infty} \frac{\log \left( \left[ \sqrt[p+1]{mp}
\right] 
  + \left [m \cdot \frac{1}{\left[ \sqrt[p+1]{mp} \right]^p} \right
 ] \right) }{\log m}
\end{equation}

The above limit exists and is equal with 

\begin{align*}
  &\lim_{m \to \infty} \frac{\log \left(\sqrt[p+1]{mp} +  m \cdot
    \frac{1}{(\sqrt[p+1]{mp})^p} \right) }{\log m} =
     \lim_{m \to \infty} \frac{\log \left((mp)^{\frac{1}{p+1}} +  m \cdot
      \frac{1}{(mp)^{p/p+1}} \right) }{\log m} \\
 &= \lim_{m \to \infty}\frac{\log \frac{m(p+1)}{(mp)^{p/p+1}}}{\log m} =  \lim_{m
    \to \infty}\frac{\log(m(p+1)) - 
     \log ((mp)^{p/p+1})}{\log m}  \\&=
  1 - \frac{p}{p+1} = \frac{1}{p+1}.
\end{align*}

Since trere exists a limit it follows that $ \overline{dim}_B (A)= \underline{dim}_B (A) = dim_B(A)  $.

Using the same reasoning we can calculate the graph dimension of set $ A $. What we need is to count $ g_{m}=\sum_{k=1}^m \min\{m,\#(K \cap [\frac{k-1}{m},\frac{k}{m}])\} $. The difference is that we now want the box for which $ f'(n_0) < \frac{1}{m^2}  $. From that
box on we will have m or more boxes of our grid meeting with our function whereas before that we will have less than m.
\\
Then we will calculate $ n_0 $ plus $ m $ times the box that $ n_0 $ lies in. 
Similar with the above we will have:

\begin{align*}
  n_0 = \left [ \sqrt[p+1]{m^2p} \right ] \\
  a_{n_0} \text{ lies in box } [ma_{n_0}] = [m
    \frac{1}{(\sqrt[p+1]{m^2p})^p}]
\end{align*}

And since we need to count each box until $ [ma_{n_0}] ~ m $ times
we need to calculate:
\[
  \displaystyle\lim_{m \to \infty} \dfrac{\log \left( \left[ \sqrt[p+1]{m^2p}~
\right] 
  + m \cdot \left [m \cdot \dfrac{1}{\left[ \sqrt[p+1]{m^2p}~ \right]^p}
  \right] \right) }{\log m}
\]
for $ \frac{p}{p+1} \leq \frac{1}{2} \Rightarrow p \leq 1 $ we have that when 
$ m \to \infty \Rightarrow  m a_{n_0} \to a \geq 1 $. Using the inequality 
$ m a_{n_0} < [m a_{n_0}~] < 2 m a_{n_0} $ we have  

\begin{align*}
  \lim_{m \to \infty} \frac{\log \left(\sqrt[p+1]{m^2p} +  m \cdot m \cdot
    \frac{1}{(\sqrt[p+1]{m^2p})^p} \right) }{\log m} = \\
  \lim_{m \to \infty} \frac{\log \left((m^2p)^{\frac{1}{p+1}} +  m^2 \cdot
    \frac{1}{(m^2p)^{p/p+1}} \right) }{\log m} = \\
  \lim_{m \to \infty}\frac{\log \frac{m^2(p+1)}{(m^2p)^{p/p+1}}}{\log m} = 
    \lim_{m \to \infty}\frac{\log(m^2(p+1)) - 
     \log ((m^2p)^{p/p+1})}{\log m} = \\
  2(1 - \frac{p}{p+1}) = \frac{2}{p+1}
\end{align*}

For $ \frac{p}{p+1} > \frac{1}{2} \Rightarrow p > 1 $ we have that when 
$ m \to \infty \Rightarrow  m a_{n_0} \to 0 \Rightarrow [m a_{n_0}~]
= 1 $. This gives:

\begin{align*}
  \lim_{m \to \infty} \frac{\log \left(\sqrt[p+1]{m^2p} +  m \right) }{\log m} =
\\
  \lim_{m \to \infty} \frac{\log \left((m^2p)^{\frac{1}{p+1}} +  m \right)
}{\log
    m} = \\
  \lim_{m \to \infty}\frac{\log \left( m^{\frac{2}{p+1}}(p^{\frac{1}{p+1}} +
    m^{\frac{p-1}{p+1}}) \right)}{\log m} =  
     \frac{2}{p+1} + \frac{p-1}{p+1} = 1.
\end{align*}

\end{proof}

\end{corol}
\section{construction of sets and refinement of \eqref{1}}

In this section, we are going to refine \eqref{1}, in the sense of Corollary \ref{corol002}, that for every set $X\subset [0,1]$ we have $$max\{1,2\dims{X}\} \leq \gdims{X} \leq 1+\dims{X} .$$ Furthermore we are going to prove that the new inequality is sharp, by constructing a set with $ \dims{X} = a $ and $ \gdims{X} = b $ , for every choice of $ 0<a \le 1 $ and $ b $ such that $$ max\{1,2a\} \le b \le 1+a. $$  

\begin{theorem}
\label{theorem002} 
  If a set $X$ has $ \dim{X} = a $ then $ \gdims{X} \geq 2a $
\end{theorem}
\begin{proof}
  Since $ \overline{dim_B} (X) = a $, we have that
  \[ 
    \limsup_{m \to \infty} \frac{\log N_{1/m}(X)}{\log m} = a. 
  \]
  If for each $ m $ we consider the set $ P^{m} $ that contains exactly one element of $ X $ for
  each box that intersects with $ X $ when we divide $ [0,1] $ in $ m $ boxes then $ N_{1/m}(X) = | P^{m} | $.
  Since there are no two elements of $ P^{m} $ in the same box, if we divide $ [0,1] $ in $ [\sqrt{m}] $ 
  boxes (which define intervals $ I_{k} $ for $ k \in \{1,..., [\sqrt{m}] \} $) we see that in each one of these boxes we have at most 
  $ [\sqrt{m}~] + 2 $ elements of $ P^{m} $. If that is not true then 
  the width of that box would have to be strictly larger than $ [\sqrt{m}] \cdot \frac{1}{m} \geq \frac{1}{[\sqrt{m}~]} $,
  which is \atopo. This means that $ | P^{m} \cap I_{k} | \leq [\sqrt{m}] + 2
  \leq 2 [\sqrt{m}~] $ for all $ k $. Now for the graph dimension we have
\begin{align*}
  \gdims{X} = &\limsup_{m \to \infty} \frac{\log g_m(X)}{\log m} \geq  
      \limsup_{m \to \infty} \frac{\log g_{[\sqrt{m}]}(X) }{\log [\sqrt{m}]} \\
  \geq & \limsup_{m \to \infty} \frac{\log \frac{1}{2} \sum_{k=1}^{[\sqrt{m}]} min\{2[\sqrt{m} ], \#(X \cap I_{k})\}}{\frac{1}{2} \log m} \\
  \geq & \limsup_{m \to \infty} \frac{\log \frac{1}{2} \sum_{k=1}^{[\sqrt{m}]} min\{2[\sqrt{m} ], \#(P^{m} \cap I_{k})\}}{\frac{1}{2} \log m} \\
  \geq & \limsup_{m \to \infty} \frac{\log \frac{1}{2} \sum_{k=1}^{[\sqrt{m}]}  \#(P^{m} \cap I_{k})\}}{\frac{1}{2} \log m} \\
  \geq & \limsup_{m \to \infty} \frac{\log \frac{1}{2} | P^{m} |}{\frac{1}{2} \log m} =  2 \limsup_{m \to
      \infty} \frac{\log N_{1/m}(X)}{\log m} = 2a.
\end{align*}

\end{proof}
\begin{corol}\label{corol002}
If a set $X$ has $ \dim{X} = a $ then $$ \max\{1,2a\}\leq \gdims{X} \leq a+1 .$$
\end{corol}
\begin{proof}
The proof is a straightforward combination of Theorem \ref{theorem002} and \eqref{1}.
\end{proof}

\section{ Construction of sets}
\begin{theorem}
Let $0 < a \leq 1$ and $b$ with $\max\{2a,1\}\leq b \leq a+1$, then it exists a compact set $X$ with $\dims{X}=a$ and $\gdims{X}=b.$
\end{theorem}
\begin{proof}
For $0 < a \leq 1$ and $b=a+1,$ any perfect set $X$ with $\dims{X}=a$ will do, due to Corollary \ref{corol1}.
We will do the construction only for $a>0$ and $b$ with $\max\{2a,1\}\leq b < a+1.$
Let $x_{n}=2^{n^n}$, $0\leq c< 1$ and $X_{n,i}=\left\{\frac{i}{x_{n}}-\frac{j}{x_{n+2}}, j=1,...,\left[x^{c}_{n}\right]\right\}$ where $i\in \left\{1,...,[x_{n}^a]=k_{n}\right\}.$ We also set $$X_{n}=\bigcup_{i=1}^{n=k_{n}}X_{n,i} \hspace{2pt}\text{ and }\hspace{2pt} X=\bigcup_{n=1}^{\infty}X_{n}\bigcup\{0\}.$$

For $x_{n},$ with $n$ sufficiently big we have:

\begin{flalign}\label{aba}
\begin{aligned}
&(a) \hspace{8pt} x^{a}_{n}\geq 1+x_{1}^{2}+...x_{n-1}^{2},\\
&(b)\hspace{8pt} \lim_{n\rightarrow \infty}\frac{\log{x_{n+1}}}{\log{x_{n}}}=\infty. 
\end{aligned}&&
\end{flalign}

Furthermore for $n$ sufficiently big  is easy to check the following properties:
\begin{flalign}\label{abab}\begin{split}
&(a)  \hspace{8pt}  \frac{i-1}{x_{n}}\leq \inf {X_{n,i}}\leq \sup{X_{n,i}}\leq \frac{i}{x_{n}},\\
&(b)  \hspace{8pt}  \text{diam}(X_{n,i}) \leq \frac{1}{x_{n+1}},\\
&(c)  \hspace{8pt}  \frac{x^{a}_{n}}{2 x_{n}}\leq \sup X_{n} \leq \frac{x^{a}_{n}}{x_{n}},\\
&(d) \hspace{8pt}  |X_{n}| \leq x^{2}_{n}.
\end{split}&&\end{flalign}\\
First we are going to calculate the upper box dimension of $X$ and in the sequel, its upper graph box dimension.

For $x_{n}\leq m \leq x_{n+1}$ we have
$$N_{\frac{1}{m}}(X)\leq 1+N_{\frac{1}{m}}(X_{{n+1}})+N_{\frac{1}{m}}(X_{{n}})+x_{n-1}^{2}+x_{n-2}^{2}+...+x_{1}^{2}.\\$$
Now since diam$(X_{n,i})<\frac{1}{x_{n+1}}\leq\frac{1}{m}$ we have that at most $2 k_{n} \leq 2 [{x_{n}^a}]<2x_{n}^a$ boxes intersecting $X_{n}$. Also by (\ref{abab},c) we have 
\begin{equation*}
\frac{m x^{a}_{n+1}}{2 x_{n+1}}\leq N_{\frac{1}{m}}(X_{n+1})\leq\frac{mx^{a}_{n+1}}{x_{n+1}}+1,\end{equation*}
therefore by using $(\ref{aba},a)$ we get
\begin{align*}
\frac{\log{N_{\frac{1}{m}}(X)}}{\log{m}}&<\frac{\log{\frac{mx^{a}_{n+1}}{x_{n+1}}+3x_{n}^{a}}}{\log{m}}\\
&<\max\left\{\frac{\log{2\frac{mx^{a}_{n+1}}{x_{n+1}}}}{\log{m}},\frac{\log{6 x_{n}^a}}{\log{m}}\right\}\\
&<\max\left\{\frac{\log{2}+\log{m}-\log{x^{1-a}_{n+1}}}{\log{m}},\frac{\log{6}+a\log x_{n}}{\log{m}}\right\}\\
&<\max\left\{1+ \frac{\log{2}-\log{x^{1-a}_{n+1}}}{\log{m}},\frac{\log{6}+a\log x_{n}}{\log{m}}\right\}\\
&<\max\left\{1+ \frac{\log{2}-\log{x^{1-a}_{n+1}}}{\log{x_{n+1}}},\frac{\log{6}+a\log x_{n}}{\log{x_{n}}}\right\}\\
&<\max\left\{a+\frac{\log{2}}{\log{x_{n+1}}}, a+\frac{\log{6}}{\log{x_{n}}}\right\}.
\end{align*}
Now by letting $m,x_n$ go to infinity, and by observing $(\ref{aba},b)$  we get
\begin{equation*}
\dims{X}\leq a.
\end{equation*}
To get the lower bound, we just look at scales $m=x_{n}$.

\begin{align*}
\dims{X}=&\limsup_{m\rightarrow\infty}\frac{\log N_{\frac{1}{m}}(X)}{\log{m}}\geq\limsup_{n\rightarrow\infty}\frac{\log N_{\frac{1}{x_{n}}}(X_{n})}{\log{x_{n}}}\\
\geq &\limsup_{n\rightarrow\infty}\frac{\log (\frac{x_{n}^a}{2x_{n-1}})}{\log{x_{n}}}\stackrel{(\ref{aba},b)}{=} a.
\end{align*}

Now for $g_{m}(X)$ we have
\begin{align*}
g_{m}(X)&\leq g_{m}\left(\bigcup_{i=n+2}^{\infty}X_{i}\right)+ g_{m}\left(\bigcup_{i=1}^{n+1}X_{i}\right)\leq m +\sum_{i=1}^{m+1}g_{m}(X_{i})\\
&\leq m+g_{m}(X_{n+1})+{x_{n}^{a+c}}+\sum_{i=1}^{n-1}x_{i}^{2}\stackrel{(\ref{aba},a)}{\leq} 2m+g_{m}(X_{n+1})+x_{n}^{a+c},
\end{align*}
where for $g_{m}(X_{n+1})$ we have
\begin{equation}\label{bbb}
g_{m}(X_{n+1})\leq \left\{\begin{aligned}&\frac{m^{2} x_{n+1}^{a}}{x_{n+1}}+m \hspace{8pt} &m\leq x^{\frac{1+c}{2}}_{n+1} \\ &{x_{n+1}^{a+c}} \hspace{8pt} &m\geq x^{\frac{1+c}{2}}_{n+1} \end{aligned}\right. .
\end{equation}
The first estimate comes from the fact that we have at most $N_{m}(X_{n+1})=\frac{mx_{n+1}^{a}}{x_{n+1}}+1$ boxes occupied by points of $X_{n+1}$, and we can utilize at most $m$ points in every one of these boxes, while the second comes from the fact that we have at most $k_{n+1}[x_{n+1}^{c}]<{x_{n+1}^{a+c}}$ points in $X_{n+1}$ in total.

Thus we get
\begin{align*}
&\frac{\log{g_{m}(X)}}{\log{m}}<\frac{\log\left({3m+(g_{m}(X_{n+1})-m)+x_{n}^{a+c}}\right)}{\log{m}}\\&<\max\left\{\frac{\log{6m}}{\log{m}},\frac{\log{3(g_{m}(X_{n+1})-m)}}{\log{m}},\frac{\log {3x_{n}^{a+c}}}{\log{m}}\right\}\\
&<\max\left\{1+\frac{\log{6}}{\log{m}},\frac{\log{3(g_{m}(X_{n+1})-m)}}{\log{m}},\frac{\log {3}+(a+c)\log{x_{n}}}{\log{x_{n}}}\right\}.
\end{align*}
It is easy to see from \eqref{bbb} that 
\begin{equation*}
\frac{\log{(g_{m}(X_{n+1})-m)}}{\log{m}}\leq \frac{\log{x_{n+1}^{a+c}}}{\log{x^{\frac{1+c}{2}}_{n+1}}}\leq 2\frac{a+c}{1+c}.
\end{equation*}
Therefore we have \begin{align*} \frac{\log{g_{m}(X)}}{\log{m}}\leq \max\left\{1+\frac{\log{6}}{\log{m}}, 2\frac{a+c}{1+c}, a+c + \frac{\log {3}}{\log{x_{n}}}\right\},
\end{align*}
and by letting $m,x_{n}$ go to infinity, and by observing that $2\frac{a+c}{1+c}>a+c$ and recalling $(\ref{aba},b)$ we have

\begin{equation*}
\gdims{X}<\max\{1,2\frac{a+c}{1+c}\}.
\end{equation*}

To get the lower bound, we just look at scales $m=\left[x^{\frac{1+c}{2}}_{n+1}\right]$. First we need to observe that since $\frac{{i}}{x_{n+1}}\in X_{n+1,i}$ we have that for every $0\leq j \leq [\frac{x^{\frac{c+2a-1}{2}}_{n+1}}{2}],$ is true that $\left[\frac{j}{\left[x^{\frac{1+c}{2}}_{n+1}\right]},\frac{j+1}{\left[x^{\frac{1+c}{2}}_{n+1}\right]}\right]$ intersects at least $\left[{x_{n+1}^{\frac{1-c}{2}}}\right]-2\geq \left[\frac{x_{n+1}^{\frac{1-c}{2}}}{2}\right]$ of the sets $X_{n+1,i}$, thus containing at least $\left[\frac{x_{n+1}^{\frac{1-c}{2}}}{2}\right]-2\geq \left[\frac{x_{n+1}^{\frac{1-c}{2}}}{4}\right]$ of them, and therefore containing at least $\left[\frac{x_{n+1}^{\frac{1-c}{2}}}{4}\right][x_{n+1}^{c}]\geq\left[\frac{x_{n+1}^{\frac{1+c}{2}}}{8}\right]$ points.  Therefore $g_{\left[x^{\frac{1+c}{2}}_{n+1}\right]}(X_{n+1})> \left[\frac{x_{n+1}^{\frac{1+c}{2}}}{8}\right]([\frac{x^{\frac{c+2a-1}{2}}_{n+1}}{2x_{n}}]+1)>\frac{x_{n+1}^{a+c}}{16x_{n}}.$

Now we have

\begin{align*}
\gdims{X}=&\limsup_{m\rightarrow\infty}\frac{\log g_{m}(X)}{\log{m}}\geq\limsup_{n\rightarrow\infty}\frac{\log g_{\left[x^{\frac{1+c}{2}}_{n+1}\right]}(X_{n+1})}{\log\left({\left[x^{\frac{1+c}{2}}_{n+1}\right]}\right)}\\
\geq &\limsup_{n\rightarrow\infty}\frac{\log\left(\frac{x_{n+1}^{a+c}}{16}\right)}{\log{\left( x^{\frac{1+c}{2}}_{n+1}\right)}}\geq 2\frac{a+c}{1+c}.
\end{align*}
Also $\gdims{X}\geq 1$ trivially. Therefore \begin{equation*}
\gdims{X}=\max\{1,2\frac{a+c}{1+c}\}.
\end{equation*}
Now by choosing $c$ such that $b=2\frac{a+c}{1+c}$ we get our result.
\end{proof}

\appendix
\section{}
 Here are some general results regarding functions in $ \IR $ and box counting that we use for the proof of Theorem \ref{theorem001}. We will consider the $ \delta $-meshes as the union $ \cup_{i,j \in \IN} {B_i^j} $ with $ B_i^j = [i\delta,(i+1)\delta) \times [j\delta,(j+1)\delta) $. That means that $ \mathbf{B}_i = \cup_{j \in \IN} B_i^j $ is the $ i+1 $ column of the mesh.

\begin{lemma}
\label{lemma001}
  
  Let $\delta>0$ and $f,g \in C_{u}(X)$ with  $g>0$ and  $||g||_{\infty} \leq \delta $. We have $N_\delta (f+g) \ge \frac{1}{2} N_\delta (f).$  
\end{lemma}
\begin{proof}
For every $x\in X$ and $B_{i}^{j}$ we have that if $ (x,f(x)) \in B_i^j $ then $(x,(f+g)(x)) \in B_i^j \cup B_{i}^{j+1}$. Now let  $\mathbf{B}_i = \cup_{j \in \IN} B_i^j $ be an arbitrary column, and $B_i^{j_{1}}, ... B_i ^{j_{i}}$ be the boxes in that column intersected from the graph of  $f(x).$ Finally let $(x_{j_1},f(x_{j_1})),...,(x_{j_i},f(x_{j_i})),$ be the points in the corresponding boxes. Wlog we can assume that $i$ is even number. Then $(x_{j_2},(f+g)(x_{j_2})),(x_{j_4},(f+g)(x_{j_4})),...,(x_{j_i},(f+g)(x_{j_i}),$ belong to different boxes. So $f+g$ intersects with at least $\frac{j_i}{2}$ boxes of the $\mathbf{B}_i$ column. Now by summing over all columns, we get what we want.
\end{proof} 
\begin{lemma}
\label{boxes_of_f+g}
If given a $ \delta $-grid and two functions f, g $ (\mathbb R \rightarrow \mathbb R^+) $ such that g intersects with $ N_\delta(g) $ boxes of the grid and f intersects with at most $ n_f $ boxes at each column of the grid then their sum intersects with at least $ \frac{N_\delta(g)}{2n_f} $ boxes of the grid.
\end{lemma}
\begin{proof}
We will first prove the result for a single column of boxes.

Let $ a_m=(x_m, g(x_m)) $ be $ n_{i,g} $ distinct points in which $ g $ intersects with the elements of the column $ \mathbf{B}_{i} $. Let $ n_{i,f} , n_{i,f+g} $ be the number of boxes of column $ i $ that intersect with $ f , f+g $ respectively.

Every $ a_m $ lies in a unique $ B_i^j $ since we are using disjoint boxes.

Since $ \forall x \in [i\delta,(i+1)\delta),$ we have that $f(x) \in B_i^j $ for some $ j \in \mathbb{N} $ and since $ f $ intersects with only $ n_{i,f} $ elements of the column it follows that there is a subset $ G $ of $ \mathbb{N} $ with exactly $ n_{j,f} $ elements such that $ \forall x \in [j\delta,(j+1)\delta)~~~f(x) \in B_i^j $ and $ i \in G $.

Now we will show that if $ (a,b) \in B_i^j $ and $ (a,c) \in B_i^l $ then $ (a,b+c) $ lies in either $ B_{i}^{j+l} $ or in $ B_{i}^{j+l+1} $.

\begin{align*}
& a \in [i\delta,(i+1)\delta] ~and~ (a,b) \in B_i^j \Leftrightarrow b \in [j,j+1) \\
& so \\
& j \leq b < j+1 \\
& l \leq c < l+1  \Rightarrow \\
& j+l \leq b+c < j+l+2 \Rightarrow \\
& b+c \in [j+l,j+l+1) ~\text{ or }~ b+c \in [j+l+1,j+l+2) \Rightarrow \\
& (a,b+c) \in B^{j+l}_i \text{ or } (a,b+c) \in B^{j+l+1}_i
\end{align*}

Now we can show that $ f+g $ intersects with at least $ \frac{n_{i,g}}{2n_f} $
elements of the column. 
If $ n_{i,f+g} < \frac{n_{i,g}}{2n_{i,f}} $ then the $ n_{i,g} $ points $ (x_m, (f+g)(x_m)) $ (where $ x_m $ are the first coordinates of the points $ a_m $) lie in less than $ \frac{n_{i,g}}{2n_{i,f}} $ elements of the column. So there must be a set of at least $ 2n_{i,f} + 1 $ of the $ x_{m} $ (we will call them $ x_{m_l} $) for which all points $
(x_{m_l}, (f+g)(x_{m_l})) $ lie in $ B_i^k $, for some $ k \in \mathbb{N} $.  

The above is assuming that $ n_{i,g} \ge 2n_{i,f} + 1 $, if this is not true we have the trivial case where $ n_{i,f+g} \ge 1 $ which is true.

Since all $ (x_{m_l},f(x_{m_l})) $ lie in at most $ n_{i,f} $ elements of $ \mathbf{B}^i $ then in the $ 2n_{i,f} + 1 $ of them there are at least 3 points $ (x_{m_l},f(x_{m_l})) $
that lie in $ B_i^n $, for some $ n \in \mathbb{N} $.

If we combine this with the above and the fact that the corresponding points $
(x_{m_l},(f+g)(x_{m_l})) $ are all in $ B_i^k $ we have 2 distinct points $
a_i=(x_i,g(x_i)) $ in the same box (either $ B_i^{k-n} $ or $ B_i^{k-n-1} $). 
Given the selection of $ \{ a_m \} $ \atopo.

Since $ n_{i,f} \le n_f $ we have $ n_{i,f+g} \ge \frac{n_{i,g}}{2n_f} $. By summing over all j (the sum is finite) we obtain $ N_\delta(f+g) \ge \frac{N_\delta(g)}{2n_f} $.

\end{proof}

\begin{lemma}
\label{polyf}
Let $ B_i^j $ a box covering of $ [0,1] \times \IR$. 
Let also $ f $ a piecewise smooth function in $ [0,1] $. Let also assume that the derivative, where it is defined, is bounded by some constant  $ k $. 
Then $ f $ meets with each column of boxes $ B_i^j $ in at most $ k+1 $ boxes.
\end{lemma}

\end{document}